\documentclass[12pt,a4paper,reqno]{amsart}
\pdfoutput=1
\usepackage[hidelinks,pdfstartview=FitH]{hyperref} 
\usepackage{amssymb}
\usepackage{tikz}
\usepackage[all,cmtip]{xy}
\usepackage[USenglish]{babel}
\usepackage[top=3cm,right=2.5cm,left=2.5cm,bottom=3cm,marginparsep=15pt,marginparwidth=2cm,footskip=20pt]{geometry}

\allowdisplaybreaks[4]

\newtheorem{prop}{Proposition}[section]
\newtheorem{lemma}[prop]{Lemma}

\newtheorem{thm}[prop]{Theorem}

\newtheorem{definition}[prop]{Definition}
\newtheorem{ex}[prop]{Example}
\newtheorem*{thm*}{Theorem}

\numberwithin{equation}{section}

\renewcommand{\emptyset}{\varnothing}
\newcommand{\id}{\mathrm{id}}

\title[Non-surjective pullbacks of graph C*-algebras]{\vspace*{-25mm}
Non-surjective pullbacks of graph C*-algebras\\ from non-injective pushouts of graphs}

\author[A.~Chirvasitu]{Alexandru Chirvasitu} 
\address[A.~Chirvasitu]{SUNY, Buffalo, USA.}
\email{achirvas@buffalo.edu}

\author[P. M.~Hajac]{Piotr M.~Hajac}
\address[P. M.~Hajac]{Instytut Matematyczny, Polska Akademia Nauk, ul. \'Sniadeckich 8, Warszawa, 00-656 Poland;
and
Department of Mathematics,
University of Colorado Boulder,
2300 Colorado Avenue,
Boulder, CO 80309-0395,
USA}
\email{pmh@impan.pl }

\author[M.~Tobolski]{Mariusz Tobolski}
\address[M.~Tobolski]{Instytut Matematyczny, Polska Akademia Nauk, ul. \'Sniadeckich 8, Warszawa, 00-656 Poland}
\email{mtobolski@impan.pl }

\begin{document}
\baselineskip14pt
\parskip5mm
\begin{abstract}
We find a substantial class of pairs of 
$*$-homomorphisms between graph C*-algebras of the form
$\xymatrix@1@C=1pc{C^*(E)\,\ar@{^{(}->}[r] &C^*(G)&\ar@{->>}[l]C^*(F)}$
whose pullback C*-algebra is an AF graph C*-algebra.
Our result can be interpreted as a recipe for determining the quantum space obtained by shrinking a quantum subspace.
There is a variety of examples from noncommutative
topology, such as
quantum complex  projective spaces (including the standard Podle\'s quantum sphere) or
quantum teardrops, that instantiate the result. Furthermore, to go beyond AF graph C*-algebras,
we consider extensions of graphs over sinks
and prove an analogous theorem for the thus obtained graph C*-algebras.
\end{abstract}
\maketitle

\section{Introduction}
\noindent
The classical two-sphere $S^2$ can be obtained by shrinking the boundary of the disc $B^2$ to a~point.
In other words, there is a pushout diagram in the category of topological spaces
\begin{equation}\label{sphere}
\begin{gathered}
\xymatrix{
& S^2 \\
\{\ast\} \ar[ru]
& 
& B^2~.\ar[lu]\\
& S^1 \ar[lu] \ar[ru]
}
\end{gathered}
\end{equation}
Due to the contravariant duality of algebras and spaces, the diagram~\eqref{sphere} amounts to
an isomorphism $C(S^2)\cong C(B^2)\oplus_{C(S^1)}\mathbb{C}$ of C*-algebras of complex-valued continuous functions on the 
two-sphere and the pushout $B^2\sqcup_{S^1}\{\ast\}$ respectively. 

At the same time, the Toeplitz algebra $\mathcal{T}$~\cite{la-c67}
can be viewed as a noncommutative deformation of $C(B^2)$ (see~\cite[Theorem~IV.7]{kl-92}).
Therefore, the C*-algebra $C(S^2_{q0})$ of the standard Podle\'s sphere~\cite[(3a)]{pod-87} 
provides a noncommutative deformation of the
diagram~\eqref{sphere}, namely we have the following pullback diagram in the category of \mbox{C*-alge}\-bras
\begin{equation}\label{podpull}
\begin{gathered}
\xymatrix{
& C(S^2_{q0}) \ar[ld] \ar[rd] \\
\mathbb{C}  \ar[rd]
& 
& \mathcal{T}~. \ar[ld]\\
& C(S^1)
}
\end{gathered}
\end{equation}
\vspace{-5mm}

The aim of this paper is to generalize the above pullback construction
using the concept of a C*-algebra $C^*(E)$ of a directed graph $E$ (e.g.,~see~\cite{bhrsz-02}).
Graph C*-algebras provide powerful tools in noncommutative topology,
and many \mbox{C*-al}ge\-bras representing noncommutative deformations of topological spaces 
are isomorphic with C*-algebras of graphs~\cite{bsz-18,hsz-02,hsz-08}. 
These isomorphisms are usually quite complicated and they do not depend on the
deformation parameter. Nevertheless, when such an isomorphism is established, it is easier to obtain solutions to
many problems, especially concerning K-theory.

Our starting point is that all the C*-algebras in the diagram~\eqref{podpull} can be viewed as C*-algebras of graphs.
We present this pictorially as follows:
\begin{equation}\label{pictpull}
\begin{gathered}
\xymatrix{
&
C^*\left(\begin{tikzpicture}[auto,swap]
\tikzstyle{vertex}=[circle,fill=black,minimum size=3pt,inner sep=0pt]
\tikzstyle{edge}=[draw,->]
\tikzstyle{cycle1}=[draw,->,out=130, in=50, loop, distance=40pt]
\tikzstyle{cycle2}=[draw,->,out=135, in=45, loop, distance=65pt]
\node[vertex] (0) at (0,0) {};
\node[vertex] (2) at (1.5,0) {};
\path (0) edge[edge,above] node {$(\infty)$} (2);
\end{tikzpicture}\right)~~~
\ar[rd] \ar[ld]&\\
C^*\left(~
\begin{tikzpicture}[auto,swap]
\centering
\tikzstyle{vertex}=[circle,fill=black,minimum size=3pt,inner sep=0pt]
\tikzstyle{edge}=[draw,->]
\tikzstyle{cycle1}=[draw,->,out=130, in=50, loop, distance=40pt]
\tikzstyle{cycle2}=[draw,->,out=135, in=45, loop, distance=65pt]
\node[vertex] (0) at (0,0) {};
\end{tikzpicture}~\right) \ar[rd]
& & 
C^*\left(\begin{tikzpicture}[auto,swap]
\tikzstyle{vertex}=[circle,fill=black,minimum size=3pt,inner sep=0pt]
\tikzstyle{edge}=[draw,->]
\tikzstyle{cycle1}=[draw,->,out=130, in=50, loop, distance=20pt]
\tikzstyle{cycle2}=[draw,->,out=130, in=50, loop, distance=70pt] 
\node[vertex] (0) at (0,0) {};
\node[vertex] (1) at (0.5,0) {};
\path (0) edge[cycle1] node {} (0);
\path (0) edge[edge] node {} (1);
\end{tikzpicture}\right)~.
\ar[ld]
\\
&
C^*\left(\begin{tikzpicture}[auto,swap]
\tikzstyle{vertex}=[circle,fill=black,minimum size=3pt,inner sep=0pt]
\tikzstyle{edge}=[draw,->]
\tikzstyle{cycle1}=[draw,->,out=130, in=50, loop, distance=20pt]
\tikzstyle{cycle2}=[draw,->,out=130, in=50, loop, distance=70pt]
\node[vertex] (0) at (0,0) {};
\path (0) edge[cycle1] node {} (0);
\end{tikzpicture}\right)
&
}
\end{gathered}
\end{equation}
(See the examples in Section~\ref{prem} for details.)

The graph-algebraic decomposition \eqref{pictpull}
manifests a certain general phenomenon
that can be explained in terms of non-injective pushouts of graphs.
The goal of this paper is to explore this phenomenon to arrive at a general setting. To this end, we search for a new concept of 
morphisms of graphs, so as to ensure that, in the thus defined category of graphs, the assignment of graph algebras to graphs 
becomes a contravariant functor translating pushouts of graphs into pullbacks of graphs algebras. While this task seems to be 
completed in \cite{hrt} (cf.~\cite[Corollary~3.4]{kpsw16}) for injective pushouts of row-finite graphs (each vertex emits only finitely 
many edges), herein we handle a non-injective case without row-finiteness assumption.

To accommodate this naturally occuring non-injectivity, we replace the standard idea of mapping vertices to vertices and edges to 
edges by the more flexible idea of mapping finite paths to finite paths. We arrive at a general result for a class of unital AF graph 
\mbox{C*-alge}\-bras including the standard Podle\'s sphere, complex quantum projective 
spaces~\cite[Definition on p.~109]{vs90}, and quantum teardrops~\cite{bf-12}. 
Finally, we go beyond AF graph \mbox{C*-alge}\-bras 
by extending their acyclic graphs over sinks. 

\section{Graph-algebraic preliminaries}\label{prem}
\noindent
A {\em directed graph} $E$ is a quadruple $(E^0,E^1,s,r)$, where $E^0$ is the set of vertices, $E^1$ is  the set of edges (arrows),
and  $s,r:E^1\to E^0$ are the source map and the range (target) map respectively.
Throughout the paper, we consider only directed graphs with countable sets of vertices and edges, and we will often 
simply refer to them as graphs.

\begin{definition}[Graph C*-algebra]
The \emph{graph \mbox{C*-algebra}} $C^*(E)$ of a directed graph $E$ is the universal \mbox{C*-algebra} generated by mutually 
orthogonal projections $\big\{P_v\;|\;v\in E^0\big\}$ and partial isometries $\big\{S_e\;|\;e\in E^1\big\}$ 
satisfying the following conditions:
\begin{align}
S_e^*S_f &=\delta_{e,f}P_{r(e)} && \text{for all }e,f\in E^1\,,\tag{\text{GA1}} \label{eq:GA1} \\
\sum_{e\in s^{-1}(v)}\!\! S_eS_e^*&=P_v && 
\text{for all }v\in E^0\text{ such that }0<|s^{-1}(v)|<\infty\,,\tag{GA2} \label{eq:GA2}\\
S_eS_e^*&\leq P_{s(e)} && \text{for all }e\in E^1.\tag{GA3}\label{eq:GA3}
\end{align}
\end{definition}

A vertex $v$ in $E$ is called a {\em sink} if and only if $s^{-1}(v)=\emptyset$. A vertex is called \emph{regular} iff
it is not a sink and it emits finitely many edges. A graph is called \emph{row finite} iff all its vertices are either regular or sinks.
By a finite \emph{path} in $E$ we mean a sequence $(e_1,\ldots,e_n)$ of edges satisfying $r(e_i)=s(e_{i+1})$ for all $i\in\{1,\ldots,n-1\}$.
The length of a path is the number of edges in the sequence. We consider vertices as paths of length zero, and
denote the set of finite paths by ${\rm Path}(E)$.
The notation $S_{\alpha}$, along with the source and the range map, naturally extend to any~$\alpha\in{\rm Path}(E)$. As we consider only
finite paths throughout this paper, we will simply refer to them as paths.

A~path $\alpha$ is called a {\em loop} if and only if $s(\alpha)=r(\alpha)$ and $\alpha$ is not a vertex. We say that a loop is
\emph{short} iff it is an edge.
\begin{definition}
We call a path {\em pointed} iff its final edge is not a loop.
\vspace*{-5mm}\end{definition}\noindent
We say that a path $\alpha$ is a {\em prolongation} of a path $\beta$ if and only if
$\alpha=\beta \gamma$ for a path $\gamma\in{\rm Path}(E)$ such that $r(\beta)=s(\gamma)$.
We write $\beta\preceq\alpha$ when $\alpha$ is a prolongation of $\beta$.
Observe that $\preceq$ gives a partial order on ${\rm Path}(E)$.  %and $\beta\not\preceq\alpha$ when it is not the case.

\begin{lemma}\label{lem:2.3}
Let $\alpha$ and $\beta$ be finite paths in an arbitrary graph $E$. Then
\begin{equation*}
S^*_\alpha S_\beta\neq 0\quad\iff\quad \left(\alpha\preceq\beta\quad\text{or}\quad\beta\preceq\alpha\right).
\end{equation*}
\end{lemma}
\begin{proof}
Assume that $\beta\preceq\alpha$, i.e.\ that $\alpha=\beta\gamma$ with $r(\beta)=s(\gamma)$. Then,
as $S^*_\gamma$ is an element of a linear basis of $C^*(E)$ for any $\gamma\in {\rm Path}(E)$~\cite[Corollary~1.5.12]{aasm-17},
we obtain
\begin{equation}
S^*_\alpha S_\beta
=S^*_\gamma S^*_\beta S_\beta=S^*_\gamma\neq 0.
\end{equation}
Much in the same way, we see that $S^*_\alpha S_\beta\neq0$ when $\alpha\preceq\beta$.

Conversely, assume that $S^*_\alpha S_\beta\neq 0$ for some finite paths 
\begin{equation}
\alpha:=x_1\ldots x_m\,,\qquad \beta:=y_1\ldots y_r\,,\qquad x_1,\ldots, x_m,y_1,\ldots, y_r\in E^1\,.
\end{equation}
If $m\geq r$, then
\begin{equation}
0\neq S^*_\alpha S_\beta=S^*_{x_m}\ldots S^*_{x_1}S_{y_1}\ldots S_{y_r}
\end{equation}
and \eqref{eq:GA1} imply that $x_i=y_i$ for $i=1,\ldots,r$. This means that $\beta\preceq\alpha$.
Otherwise, when $r\geq m$, we get that $\alpha\preceq\beta$.
\end{proof} 

Next, to make the  condition \eqref{eq:GA3} easier to check, we prove the following lemma:
\begin{lemma}\label{specon}
  Let $E$ be an arbitrary graph and $\alpha$ a path in $E$ with its origin at $v\in E^0$. Then
  \begin{equation*}
    S_\alpha S_\alpha^*\le P_v\in C^*(E). 
  \end{equation*}  
\end{lemma}
\begin{proof}
  Write $\alpha=\beta e$, where  $e$ is  an edge with its origin at $w\in E^0$ and $\beta$ is an initial subpath  of $\alpha$ ending at~$w$. 
Then
  \begin{equation}
    S_\alpha S_\alpha^* = S_\beta (S_eS_e^*)S_\beta^*\le S_\beta P_w S_\beta^* = S_\beta S_\beta^*\,,
  \end{equation}
  where the middle inequality is due to~\eqref{eq:GA3}. Now, the claim follows by the induction on the length of~$\alpha$.
\end{proof}

To get ready for examples in the last section, we present graph-algebraic presentations of some well-known \mbox{C*-al}ge\-bras.
\begin{ex}
 \emph{The algebra $\mathbb{C}$ of complex numbers is isomorphic with the graph C*-algebra
of the graph with one vertex and no edges.
\begin{equation*}
\begin{tikzpicture}[auto,swap]
\tikzstyle{vertex}=[circle,fill=black,minimum size=3pt,inner sep=0pt]
\tikzstyle{edge}=[draw,->]
\tikzstyle{cycle1}=[draw,->,out=130, in=50, loop, distance=40pt]
\tikzstyle{cycle2}=[draw,->,out=135, in=45, loop, distance=65pt]
\node[vertex,label=right:$v$] (0) at (0,0) {};
\end{tikzpicture}
\end{equation*}}
\end{ex}
\begin{ex}%[Continuous complex-valued functions on the circle $C(S^1)$]
 \emph{The C*-algebra $C(S^1)$ of all continuous complex-valued functions on the circle is the
universal unital C*-algebra generated by a single unitary~$u$. It 
is isomorphic with the graph C*-algebra of the graph %$L$ 
given below
through the isomorphism given by~\mbox{$u\mapsto S_e$}.
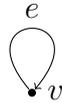
\begin{equation*}
\begin{tikzpicture}[auto,swap]
\tikzstyle{vertex}=[circle,fill=black,minimum size=3pt,inner sep=0pt]
\tikzstyle{edge}=[draw,->]
\tikzstyle{cycle1}=[draw,->,out=130, in=50, loop, distance=40pt]
\tikzstyle{cycle2}=[draw,->,out=135, in=45, loop, distance=65pt]
\node[vertex,label=right:$v$] (0) at (0,0) {};
\path (0) edge[cycle1,above] node {$e$} (0);
\end{tikzpicture}
\end{equation*}}
\end{ex}
\begin{ex} \emph{The Toeplitz algebra $\mathcal{T}$~\cite{la-c67} is
the universal unital C*-algebra generated by a single isometry~$s$. It
is isomorphic with the graph C*-algebra of the graph %$T$ 
given below
through the isomorphism given by $s\mapsto S_{t_1}+S_{t_2}$.
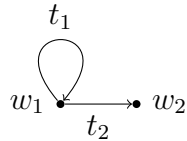
\begin{equation*}\label{toe}
\begin{tikzpicture}[auto,swap]
\tikzstyle{vertex}=[circle,fill=black,minimum size=3pt,inner sep=0pt]
\tikzstyle{edge}=[draw,->]
\tikzstyle{cycle1}=[draw,->,out=130, in=50, loop, distance=40pt]
\tikzstyle{cycle2}=[draw,->,out=135, in=45, loop, distance=65pt]
\node[vertex,label=left:$w_1$] (0) at (0,0) {};
\node[vertex,label=right:$w_2$] (1) at (1,0) {};
\path (0) edge[cycle1, above] node {$t_1$} (0);
\path (0) edge[edge, below] node {$t_2$} (1);
\end{tikzpicture}
\end{equation*}}
\end{ex}
\begin{ex}\emph{The Cuntz algebra $\mathcal{O}_m$ \cite{cuntz} is
the universal unital C*-algebra generated by isometries $s_1$,~$\ldots$, $s_m$ subject to the relation
$\sum_{i=1}^ms_is^*_i=1$. It is isomorphic with the graph \mbox{C*-algebra} of the graph $R_m$ given below
through the isomorphism given by $s_i\mapsto S_{e_i}$.
\begin{equation}\label{cuntz}
\begin{tikzpicture}[auto,swap]
\tikzstyle{vertex}=[circle,fill=black,minimum size=3pt,inner sep=0pt]
\tikzstyle{edge}=[draw,->]
\tikzstyle{cycle1}=[draw,->,out=130, in=50, loop, distance=40pt]
\tikzstyle{cycle2}=[draw,->,out=130, in=50, loop, distance=70pt]
\node[vertex,label=right:$1$] (0) at (0,0) {};
\node (1) at (0,1.25) {\vdots};
\path (0) edge[cycle1] node {$e_1$} (0);
\path (0) edge[cycle2] node[above] {$e_m$} (0);
\end{tikzpicture}
\end{equation}}
\end{ex}
\begin{ex}\emph{
Let $q\in [0,1)$.  
The \mbox{C*-alge}\-bra $C(S^2_{q0})$~\cite[(3a)]{pod-87} of the standard Podle\'s quantum sphere  
coincides with the C*-algebra of the Vaksman--Soibelman quantum complex projective line 
$C(\mathbb{C}{\rm P}^1_q)$ \cite[p.~109]{vs90},
%(the minimal unitization of the C*-algebra of compact operators $\mathcal{K}$) 
which has a graph-algebraic presentation as the graph C*-algebra of the graph given 
below (see \cite[Section~2.3]{hsz-02}):
\begin{equation*}\label{podgraf}
\begin{tikzpicture}[auto,swap]
\tikzstyle{vertex}=[circle,fill=black,minimum size=3pt,inner sep=0pt]
\tikzstyle{edge}=[draw,->]
\tikzstyle{cycle1}=[draw,->,out=130, in=50, loop, distance=40pt]
\tikzstyle{cycle2}=[draw,->,out=135, in=45, loop, distance=65pt]
\node[vertex,label=left:$v_1$] (0) at (0,0) {};
%\node (1) at (1,0) {$(\infty)$};
\node[vertex,label=right:$v_2$] (2) at (2,0) {};
%\path (0) edge[cycle1] node {$s_1$} (0);
%\path (0) edge[draw] node {} (1);
\path (0) edge[edge, above] node {$(\infty)$} (2);
%\path (0) edge[edge,bend right] node {} (1);
%\path (0) edge[edge] node {} (2);
%\path (0) edge[cycle2] node {$s_2$} (0);
%\path (2) edge[edge] node {} (1);
\end{tikzpicture}
\end{equation*}
Here the arrow decorated by $(\infty)$ denotes  countably infinitely many arrows.}
\end{ex}

We end this section by recalling some standard results that we will use throughout the paper.
Let $E$ be a directed graph.
A subset $H\subseteq E^0$ is called {\em hereditary} iff, for any $v\in H$ such that there 
is a path starting at $v$ and ending at $w\in E^0$, we have $w\in H$. 
If $H$ is hereditary, then the ideal $I_H$ generated by the projections associated with the elements of $H$ is of the form 
(cf.\ the equation (1) in~\cite{bhrsz-02}):
\begin{equation}\label{hereditary}
I_H=\overline{\rm span}\{S_\alpha S^*_\beta~|~\alpha,\beta\in{\rm Path}(E),r(\alpha)=r(\beta)\in H\}.
\end{equation}
Here $\overline{\rm span}$ denotes the closed linear span.

Assume additionally that there are no vertices that emit infinitely many arrows into $H$ 
and finitely many (but not zero) arrows
outside of $H$. Assume also that $H$ is {\em saturated}, i.e.\ that there does not
exist a regular vertex $v\notin H$ such that
$r(s^{-1}(v))\subseteq H$.
Then, the quotient algebra $C^*(E)/I_H$ is again a graph C*-algebra (cf.\ the discussion below the equation (1) in \cite{bhrsz-02}):
\begin{equation}\label{quotient}
C^*(E)/I_H\cong C^*(E/H),\quad\text{where}\quad E/H:=(E^0\setminus H,r^{-1}(E^0\setminus H),s_H,r_H)
\end{equation}
and $s_H$ and $r_H$ are the restrictions-corestrictions of $s$ and $r$ respectively.

\section{Non-surjective pullbacks of graph C*-algebras}
\noindent
In this section we prove a non-surjective pullback theorem generalizing the diagram~\eqref{podpull}.
First, we need some preliminaries on graphs and their morphisms.

Let $D=(D^0,D^1,s_D,r_D)$ and $E=(E^0,E^1,s_E,r_E)$ be directed graphs. A~morphism of graphs $f:D\to E$
is a pair of mappings $f^0:D^0\to E^0$ and $f^1:D^1\to E^1$ satisfying
\begin{equation}\label{graphmor}
f^0\circ s_D=s_E\circ f^1,\qquad f^0\circ r_D=r_E\circ f^1.
\end{equation}
If there is an injective morphism of graphs $D\to E$, we say that $D$ is a {\em subgraph} of $E$ and write $D\subseteq E$.
\begin{definition}\label{adm}
An injective graph morphism $\iota:D\to E$ is called an {\em admissible inclusion} iff the following conditions are satisfied:
\vspace*{-5mm}\begin{enumerate}
\item[(A1)] $E^0\setminus \iota^0(D^0)$ is hereditary and saturated,
\item[(A2)] $\iota^1(D^1)=r_E^{-1}(\iota^0(D^0))$,
\item[(A3)] no vertex in $E^0$ emits infinitely many edges into 
$E^0\setminus D^0$ while emitting finitely many (but not zero) edges into $D^0$.
\end{enumerate}
%If there is an admissible inclusion $D\to E$ and 
%we call the pair $(D\subseteq E)$ an {\em admissible graph}.
\end{definition}

Next, let us state the following elementary fact (cf.~\eqref{quotient} and the discussion preceding~it).
\begin{prop}\label{admquot}
Let $D\subseteq E$ be an admissible inclusion. 
Then, we have an isomorphism of graph C*-algebras
\begin{equation}\label{3.2}
C^*(D)\cong C^*(E/(E^0\setminus D^0)).
\end{equation}
\end{prop}

To phrase our main result, it is convenient to view graphs as small categories whose objects are vertices and morphisms are finite paths.
Then functors between such categories are  what we want as morphisms between graphs. Using the thus understood functors as morphisms, we 
 generalize the Cuntz--Krieger graph category \cite[p.~172]{g-kr09} (cf.~\cite[Definition~1.6.2]{aasm-17})
by allowing egdes to be mapped to finite paths intead of only edges.
\begin{lemma}\label{functor}
Let $f\colon F\to E$ be a functor between graphs such that:
\vspace*{-5mm}
\begin{enumerate}
\item
$f$ is compatible with the prolongation relation as follows
\begin{gather*}
f(\alpha)\preceq f(\beta)\quad\Rightarrow\quad \alpha\preceq\beta;
\end{gather*}
\item
for any vertex $v$ that emits at least one and at most finitely many edges, $f$  restricts-corestricts to a bijection
\[
s_F^{-1}(v)\longrightarrow s_E^{-1}(f(v)).
\]
\end{enumerate}
Then $f$ induces a $*$-homomorphism $f_*\colon C^*(F)\to C^*(E)$ given by 
\[
\forall\;v\in F^0\colon\; f_*(P_v):=P_{f(v)}\quad\text{and}\quad \forall\;x\in F^1\colon\;f_*(S_x):=S_{f(x)}\;.
\]
\end{lemma}
\begin{proof}
Since graph C*-algebras are universal, it suffices to show that all defining relations are preserved. For starters, since the condition (1)
implies the injectivity of $f$, we infer that 
 the set of mutually orthogonal projections is sent to the set of mutually orthogonal projections:
\begin{equation}
P_{f(v)}P_{f(w)}=\delta_{f(v),f(w)}P_{f(v)}=\delta_{v,w}P_{f(v)}.
\end{equation}
Next,
to show that \eqref{eq:GA1} is preserved, it suffices to prove the implication
\begin{equation}
S^*_{f(e_1)}S_{f(e_2)}\neq 0\quad\Rightarrow\quad e_1=e_2\,,
\end{equation}
which follows from combining Lemma~\ref{lem:2.3} with the condition (1).
Finally, showing that  \eqref{eq:GA2}  and \eqref{eq:GA3} are preserved is also straightforward: the former 
follows directly from the condition (2) and the latter from Lemma~\ref{specon}.
\end{proof}

We are now ready to prove our first main result:
\begin{thm}\label{main}
Let  $F_i\subseteq E_i$, $i=1,2$, be admissible inclusions of graphs such that
\begin{enumerate}
\vspace*{-5mm}\item
$E_1$ has no loops, $E_2$ has no short loops at vertices in $E_2^0\setminus F_2^0$, and $E^0_1=E^0_2$,  $F^0_1=F^0_2$;
\item
there is a functor $f\colon E_1\to E_2$ such that: 
it satisfies the condition (1) in Lemma~\ref{functor},
it is $\mathrm{id}$ on objects, and its image is the set of all pointed paths.
\vspace*{-5mm}\end{enumerate}
Then the  induced $*$-homomorphisms exist and render the diagram
\begin{equation}\label{noninj}
\xymatrix{
&C^*(E_1)\ar[dl]_{\pi_1}\ar[dr]^{f_*}&\\
C^*(F_1)\ar[dr]_{f|_*}&& C^*(E_2)\ar[dl]^{\pi_2}\\
&C^*(F_2)&
}
\end{equation}
a pullback diagram of C*-algebras. (If $E_1^0$ is finite, then this is a pullback diagram of unital C*-algebras.) Here
$\pi_1$ and $\pi_2$ are the canonical surjections~\eqref{3.2},
$f_*$ is a $*$-homomorphism of Lemma~\ref{functor}, and $f|_*$ is its restriction-corestriction.
\end{thm}
\begin{proof}
We begin by proving that $f_*$ and $f|_*$ are well-defined injective $*$-homomorphisms. 
To see that $f_*$ is well defined, by Lemma~\ref{functor} and the assumption (2), it suffices to check the condition (2) of Lemma~\ref{functor}.
To this end, take any regular vertex $v\in E^0_1$ and any edge $e \in s^{-1}(v)$. Then, as the image of $f$ is the set of pointed paths, 
$f(e)$ is a pointed path
from $v$ to~$r(e)$. 

Suppose that $f(e)$ factorizes through a third vertex~$w$. Then we can write $f(e)=\alpha\beta$, where $\alpha$ is a pointed
path from $v$ to $w$ and $\beta$ is a pointed path from $w$ to $r(e)$. Indeed, deleting any intial subpath from a pointed path  always yields
a pointed path, and making all loops based at $w$ part of $\beta$ makes $\alpha$ a pointed path. Furthermore, as $f$ is surjective on 
the set of pointed paths, we can write $f(e)=\alpha\beta=f(\alpha')f(\beta')=f(\alpha'\beta')$. Combining it with the injectivity of $f$,
which follows from the condition (1) in Lemma~\ref{functor}, we get a contradiction $e=\alpha'\beta'$ (the edge $e$ is not a path
factorizing through the vertex~$w$). Hence $f(e)$ is a pointed path from $v$ to $r(e)$ that does not factorize through any third vertex.

If there is a loop in $E_2$ based at $v$, then there are infinitely many
 non-factorizing pointed paths from $v$ to $r(e)$, and (because $f$ is a functor) none of them can be the image of a path that factorizes
 through a third vertex. Consequently, as $f$ is bijective when corestricted to the set of pointed paths,
 and there are no loops in $E_1$, there must be infinitely many edges in $E_1$
from $v$ to $r(e)$, which contradicts the assumption that $v$ is a regular vertex in~$E_1$. Hence, there is no loop in $E_2$ based at~$v$,
so $f(e)$ is an edge. 

Next, if $f(\alpha)\in E^1_2$, then $\alpha\in E^1_1$ because $f$ is an injective functor that is $\id$ on the set of veritices. Indeed,
suppose that $\alpha=e_1\dots e_n$, where $e_i$'s are edges. Then $f(\alpha)=f(e_1\dots e_n)=f(e_1)\dots f(e_n)$ is of length 
at least $n$, as $f(e_i)$ cannot be a vertex. Hence $n=1$, i.e.\ $\alpha$ is an edge, so any edge emitted from $v$ in $E_2$ comes
from an  edge emitted from $v$ in $E_1$. Combining this with the injectivity of $f$ and the above established fact that $f(e)$ is an
edge, we conclude that the condition (2) in Lemma~\ref{functor} is satisfied. 

Thus we obtain a well-defined $*$-homomorphism $f_*$ that is
injective  by~\cite[Corollary~1.3]{szy-gen} because $E_1$ has no loops. Furthermore, by the admissibility condition Definition~\ref{adm}(A2),
it is clear that $f$ restricted to the subgraph $F_1$ corestricts to $F_2$ yielding a restriction-corestriction $f|_*$  of~$f_*$. The  
$*$-homomorphism $f|_*$ is   injective because $f_*$ is injective.

It is straightforward to check that the maps $\pi_1$, $\pi_2$, $f_*$ and $f|_*$ make 
the  diagram~\eqref{noninj} commutative. Therefore, as $\pi_1$ and $\pi_2$ are surjective and $f_*$ and $f|_*$ are injective,
due to \cite[3.1 Proposition]{ped-99}, to show that \eqref{noninj} is a pullback diagram, it suffices to prove that
\begin{equation}\label{pullker2}
\ker\pi_2\subseteq f_*(\ker\pi_1).
\end{equation}
To obtain the above inclusion, we use the characterization of ideals associated to hereditary subsets~\eqref{hereditary}:
\begin{gather}
\ker\pi_1=\overline{\rm span}\left\{S_\alpha S_\beta^*~|~\alpha, \beta\in{\rm Path}(E_1),
~r(\alpha)=r(\beta)\in E_1^0\setminus F_1^0\right\},\\
\ker\pi_2=\overline{\rm span}\left\{S_\gamma S_\delta^*~|~\gamma, \delta~\in{\rm Path}(E_2),
~r(\gamma)=r(\delta)\in E_2^0\setminus F_2^0\right\}.
\end{gather}
By the assumption (1),  all paths in $E_2$ terminating in $E_2^0\setminus F_2^0$ are pointed, so they are in the image of~$f$. Therefore,
as $E_2^0\setminus F_2^0=E_1^0\setminus F_1^0$ by the assumption~(1), we conclude that the inclusion \eqref{pullker2} holds
at the algebraic level.
Finally,
as any $*$-homomorphism between C*-algebras is a continuous map whose image is closed, we infer the desired
inclusion at the C*-level.
\end{proof}

\section{Extending graphs over sinks}
\noindent
To generalize the diagram~\eqref{podpull} even further (e.g.\ to allow loops in $E_1$ in the pullback theorem of the previous section),
we first need to determine suitable conditions under which the graph-algebra construction preserves 
pushouts of graphs over sinks. 

The general setup assumptions (GS) are as follows:
\vspace*{-2mm}
\begin{itemize}
\item $E$ and $H$ are graphs;
\item $X$ is a set regarded as a graph with no edges;
\item $\iota_E:X\to E^0$ and $\iota_H:X\to H^0$ are injective maps defining the pushout
\vspace*{-2mm}
\begin{equation}\label{}
\begin{gathered}
\xymatrix{
& E^0\underset{X}{\sqcup}H^0 \\
E^0 \ar[ru]
& 
& 
H^0; \ar[lu]\\
& X\ar[lu]^{\iota_E} \ar[ru]_{\iota_H}
}
\end{gathered}
\end{equation}
\item
$E\sqcup_XH:=(E^0{\sqcup}_XH^0,E^1\sqcup H^1,\pi\circ(s_E\sqcup s_H), \pi\circ(r_E\sqcup r_H))$, where
$\pi$ is the canonical quotient map. 
\end{itemize}

Next, let  $\iota_{E*}:C^*(X)\to C^*(E)$ and $\iota_{H*}:C^*(X)\to C^*(H)$ be the induced 
\mbox{$*$-homo}\-morphisms (see Lemma~\ref{functor}). Define
\begin{equation*}
  C^*(E)\underset{C^*(X)}{\bullet} C^*(H):=(C^*(E)\underset{C^*(X)}{\ast} C^*(H))/\langle P_vP_w\;|\; v\in E^0\setminus \iota_E(X),\,
 w\in H^0\setminus \iota_H(X)  \rangle\,.
\end{equation*}
Here we divide the amalgamated free product by the ideal generated by the product of non-identified projections.

\begin{lemma}\label{le.push}
 Assume that at least one of the maps $\iota_H$ and $\iota_E$ takes its values in the sinks of the respective graph. 
 Then the natural assignment of elements defines an isomorphism of C*-algebras:
  \begin{equation}\label{eq:1}
    C^*(E)\underset{C^*(X)}{\bullet} C^*(H)\longrightarrow C^*(E\underset{X}{\sqcup} H).
  \end{equation}
\end{lemma}
\begin{proof}
 Since  $\iota_E:X\to E^0$ or $\iota_H:X\to H^0$ takes values in the sinks of $E$ or~$H$, respectively, all edge 
 relations in $C^*(E\sqcup_XH)$ involving vertices in the
  image of $X$ are of one of two types: either they refer to edges only in $E^1$, or to edges only in~$H^1$. Hence, there are
 $*$-homomorphisms 
 \begin{equation}\label{jh}
 j_E\colon C^*(E)\longrightarrow C^*(E\sqcup_X H)\quad\text{and}\quad j_H\colon C^*(H)\longrightarrow C^*(E\sqcup_X H)
 \end{equation}
 given the natural assignment of elements. Furthermore, as $(E\sqcup_X H)^0=
 E^0\sqcup_X H^0$, they induce a surjective $*$-homomorphism 
 \begin{equation}
 \pi_\sqcup\colon C^*(E)\underset{C^*(X)}{\ast}C^*(H)\longrightarrow C^*(E\underset{X}{\sqcup} H).
 \end{equation}
 Finally, as the kernel of $\pi_\sqcup$ coincides with the kernel of the defining surjection
 \begin{equation}
\pi_\bullet\colon C^*(E)\underset{C^*(X)}{\ast} C^*(H)\longrightarrow C^*(E)\underset{C^*(X)}{\bullet} C^*(H),
\end{equation}
 the claim follows.
\end{proof}

Now, consider three graphs $E_1$, $E_2$ and $H$ with injective maps $\iota_{E_1}\!:\!X\to E_1^0$, \mbox{$\iota_{E_2}\!:\!X\to E_2^0$}, 
and $\iota_{H}\!:\!X\to H^0$.  Assume also  that $\iota_{E_1}(X)$ and $\iota_{E_2}(X)$ consist of sinks of the two 
respective graphs $E_1$ and $E_2$. Now, consider a C*-algebra homomorphism
\begin{equation}\label{4.6}
  \delta:C^*(E_1)\longrightarrow A
\end{equation}
annihilating the vertex projections of $\iota_{E_1}(X)\subseteq E_1^0$. Then $\delta$ and the zero map $C^*(H)\to A$ induce
a $*$-homomorphism on the amalgamated product that annihilates the kernel of $\pi_\bullet$. Hence, by Lemma~\ref{le.push}, 
$\delta$   extends to
\begin{equation}\label{4.7}
  \delta':C^*(E_1\underset{X}{\sqcup}H)\longrightarrow A.
\end{equation}

\begin{lemma}\label{le.ker-desc}
Let $j^1_H\colon C^*(H)\to C^*(E_1\sqcup_XH) $ be the map defined in~\eqref{jh}. Then 
\[
\ker\delta'= j_{E_1}(\ker\delta)+\langle j^1_H(C^*(H))\rangle.
\]
\end{lemma}
\begin{proof}
The inclusion $\ker\delta'\supseteq\langle j_{E_1}(\ker\delta)\rangle+\langle j^1_H(C^*(H))\rangle$ 
is clear by the construction of~$\delta'$. For the other inclusion, note that,
  as $\delta'$ annihilates $C^*(H)$, it factors as
  \begin{equation}
    C^*(E_1\sqcup_X H)\to C^*(E_1\sqcup_XH)/\langle j^1_H(C^*(H))\rangle\cong C^*(E_1)/\langle \iota_{E_1*}(C^*(X))\rangle\to A,
  \end{equation}
  where the last map is induced by $\delta:C^*(E_1)\to A$. The inclusion follows from this factorization. Finally,
as 
\begin{equation}
\langle j_{E_1}(\ker\delta)\rangle= j_{E_1}(\ker\delta)+\langle j_{E_1}(\ker\delta)\rangle\cap\langle j^1_H(C^*(H))\rangle,
\end{equation}
we infer that
\begin{equation}\label{4.10}
\langle j_{E_1}(\ker\delta)\rangle+\langle j^1_H(C^*(H))\rangle=j_{E_1}(\ker\delta)+\langle j^1_H(C^*(H))\rangle,
\end{equation}
which ends the proof.
\end{proof}

Assume now that, under the general setup assumptions (GS) and the assumptions preceding~\eqref{4.6}, 
we have a pullback diagram
\begin{equation}\label{eq:2}
\begin{gathered}
\xymatrix{
& C^*(E_1) \ar[ld]_\delta \ar[rd]^\phi \\
A  \ar[rd]_\rho
& 
& C^*(E_2) \ar[ld]^\theta\\
& B
}
\end{gathered}
\end{equation}
of $*$-homomorphism of C*-algebras.
Here $\phi$ is an injective $*$-homomorphism  sending all vertex projections to vertex projections and 
all partial isometries associated with edges to partial isometries associated with paths, intertwining $\iota_{E_1}$ with 
$\iota_{E_2}$, and sending the projections labeled by $E_1^0\setminus \iota_{E_1}(X)$ to 
projections labeled by $E_2^0\setminus \iota_{E_2}(X)$.
We also assume that all partial isometries in $C^*(E_2)$ associated with paths ending in
$\iota_{E_2}(X)$ are in the image of~$\phi$. Furthermore, we assume that $\delta$ is a $*$-homomorphism
annihilating the vertex projections labeled by $\iota_{E_1}(X)$, and
$A$ and $B$ are arbitrary C*-algebras fitting into the pullback diagram~\eqref{eq:2} for some $*$-homomorphisms $\theta$
and~$\rho$.

Note that the assumptions made on $\phi$ allow us to define its extension
\begin{equation}\label{psi}
\psi:C^*(E_1\underset{X}{\sqcup}H)\longrightarrow C^*(E_2\underset{X}{\sqcup}H).
\end{equation}
Indeed, we can use the isomorphism \eqref{eq:1} and observe that the conditions on $\phi$ allow us to extend it by 
$\id\colon C^*(H)\to C^*(H)$ to
\begin{equation}\label{psi}
C^*(E_1)\underset{C^*(X)}{\bullet} C^*(H)\longrightarrow C^*(E_2)\underset{C^*(X)}{\bullet} C^*(H).
\end{equation}
This brings us to the second main result of the paper:
\begin{thm}\label{le.blob}
Under the general setup assumptions (GS) and the additional assumptions preceding \eqref{4.6},
the pullback diagram \eqref{eq:2} of $*$-homomorphisms of C*-algebras induces the following pullback diagram
of $*$-homomorphisms of C*-algebras:
\begin{equation}\label{sinkext}
\begin{gathered}
\xymatrix{
& C^*(E_1\underset{X}{\sqcup}H) \ar[ld]_{\delta'} \ar[rd]^\psi \\
A  \ar[rd]_\rho
& 
& C^*(E_2\underset{X}{\sqcup}H). \ar[ld]^{\theta'}\\
& B
}
\end{gathered}
\end{equation} 
Here $\delta'$ and $\theta'$ are defined by \eqref{4.7}, and $\psi$ is defined by~\eqref{psi}.
\end{thm}
\begin{proof}
The commutativity of the diagram \eqref{sinkext} is immediate by construction. To prove that it is a pullback diagram,
first we establish
the injectivity of $\psi$. It follows from: the injectivity of $\phi$, the assumption that $\phi$ does not annihilate vertex projections,
the fact that loops without exit in $E_1\sqcup_XH$ remain loops without exit
both in $E_1$ and $H$, and the general Cuntz--Krieger uniqueness theorem \cite[Theorem 1.2]{szy-gen}.

Next, using the injectivity of $\psi$ and appealing to \cite[3.1 Proposition]{ped-99}, we note that to conclude the proof of the theorem,
it suffices to check the following two conditions:
  \begin{gather}
\rho^{-1}\left(\theta'\left(C^*(E_2\underset{X}{\sqcup}H)\right)\right)=\delta'\left(C^*(E_1\underset{X}{\sqcup}H)\right),\\
\ker\theta'\subseteq \psi(\ker\delta').   
  \end{gather}

  The first condition is immediate from our assumption that (\ref{eq:2}) is a pullback.
Indeed, the analogous equation holds for the unprimed maps $\theta$ and 
$\delta$, and the images of these maps coincide with those of $\theta'$ and $\delta'$ respectively because the primed maps are obtained from 
the unprimed maps by extending them by the zero map on~$C^*(H)$.

 To show the second condition, we apply Lemma~\ref{le.ker-desc} and \eqref{4.10} to obtain:
  \begin{align}
    \ker\delta'&= \langle j_{E_1}(\ker\delta)\rangle+\langle j^1_H(C^*(H))\rangle,\label{17}\\
    \ker\theta'&= j_{E_2}(\ker\theta)+\langle j^2_H(C^*(H))\rangle.\label{18}
  \end{align}
Here both $j^1_H$ and $j^2_H$ are defined as in~\eqref{jh}. Now,
  since (\ref{eq:2}) is a pullback diagram, we have
 \mbox{$    \ker\theta\subseteq \phi(\ker\delta)$}. Furthermore, it follows from the construction of $\psi$ and $\delta'$ that
\begin{equation}
j_{E_2}(\phi(\ker\delta))=\psi(j_{E_1}(\ker\delta))\subseteq \psi(\ker\delta').
\end{equation}
 Hence $j_{E_2}(\ker\theta)\subseteq \psi(\ker\delta')$, so, by \eqref{18},
   we are left having to argue that
  \begin{equation}\label{whatever3}
    \langle j^2_H(C^*(H))\rangle \subseteq \psi(\ker\delta').
  \end{equation}

To this end, note first that
\begin{equation}\label{whatever}
   j^2_H(C^*(H))\subseteq \psi(\ker\delta') 
\end{equation}
 because $j^1_H(C^*(H))\subseteq \ker\delta' $ and $\psi\circ j^1_H=j^2_H$.
Furthermore, observe that
 \begin{equation}\label{whatever2}
    \langle j^2_H(C^*(H))\rangle =j^2_H(C^*(H))+\langle j^2_H(C^*(H))\rangle\cap\langle j_{E_2}(C^*(E_2))\rangle
  \end{equation}
and
\begin{equation}\label{22}
 \langle j^2_H(C^*(H))\rangle\cap\langle j_{E_2}(C^*(E_2))\rangle\subseteq\langle\{j_{E_2}(P_v)\;|\;v\in\iota_{E_2}(X)\}\rangle.
  \end{equation}

Next, since all partial isometries in $C^*(E_2)$ associated with paths ending in
$\iota_{E_2}(X)$ are in the image of~$\phi$ by assumption,
we conclude that
\begin{equation}\label{23}
\langle\{j_{E_2}(P_v)\;|\;v\in\iota_{E_2}(X)\}\rangle= \psi(\langle\{j_{E_1}(P_v)\;|\;v\in\iota_{E_1}(X)\}\rangle).
  \end{equation}
Indeed, it boils down to showing that
\begin{equation}\label{subeq}
\langle\psi(\{j_{E_1}(P_v)\;|\;v\in\iota_{E_1}(X)\})\rangle\subseteq\psi(\langle\{j_{E_1}(P_v)\;|\;v\in\iota_{E_1}(X)\}\rangle).
  \end{equation}
Since any graph C*-algebra is the closed linear span of elements of the form $S_\alpha S^*_\beta$ with $r(\alpha)=r(\beta)$
(see \cite[Corollary~1.5.12]{aasm-17}), we are looking at $\alpha,\beta\in\mathrm{Path}(E_2\sqcup_XH)$ such that $P_v S_\alpha S^*_\beta$
or $S_\alpha S^*_\beta P_v$ can be non-zero. This means that $s(\alpha)=v$ or $s(\beta)=v$. However, as $v$ is a sink in $E_2$, we infer
that $\alpha\in \mathrm{Path}(H)$ or $\beta\in \mathrm{Path}(H)$. Hence $\alpha\in \mathrm{Path}(H)$ and $r(\beta)\in H^0$,
or $\beta\in \mathrm{Path}(H)$ and $r(\alpha)\in H^0$. Furthermore, any $E_2$-subpath $\gamma$ of any path ending in $H^0$
 has to end in~$\iota_{E_2}(X)$. Consequently, $S_\gamma\in\phi(C^*(E_1))$, so $j_{E_2}(S_\gamma)\in\psi(C^*(E_1\sqcup_XH))$.
As $j^2_H(S_\delta)\in\psi(C^*(E_1\sqcup_XH))$ for any $\delta\in\mathrm{Path}(H)$, we conclude that all elements $S_\alpha S^*_\beta$
that can multiply nontrivially with $P_v$ are in the image of~$\psi$, which proves~\eqref{subeq}.

Now, taking advantage of the assumption that $\delta(P_v)=0$ for any $v\in\iota_{E_1}(X)$, we infer that
\begin{equation}
\langle\{j_{E_1}(P_v)\;|\;v\in\iota_{E_1}(X)\}\rangle\subseteq\langle j_{E_1}(\ker\delta)\rangle.
  \end{equation}
Hence, by \eqref{17},
\begin{equation}\label{25}
\psi(\langle\{j_{E_1}(P_v)\;|\;v\in\iota_{E_1}(X)\}\rangle)\subseteq\psi(\langle j_{E_1}(\ker\delta)\rangle)\subseteq\psi(\ker\delta').
  \end{equation}
Consequently, combining  \eqref{22},  \eqref{23} and \eqref{25}, we arrive at
\begin{equation}
 \langle j^2_H(C^*(H))\rangle\cap\langle j_{E_2}(C^*(E_2))\rangle\subseteq\psi(\ker\delta'),
  \end{equation}
which, together with \eqref{whatever} and \eqref{whatever2} proves~\eqref{whatever3}.
\end{proof}

\section{Examples and applications}
\noindent
This section is devoted to the study of special cases of Theorem~\ref{main} and Theorem~\ref{le.blob} 
leading to interesting examples in noncommutative topology.

\subsection{The standard Podle\'s quantum sphere}

Observe that the assumptions of Theorem~\ref{main} are true for the standard Podle\'s quantum sphere.
Here $C^*(E_1)=C(S^2_{q0})$, $C^*(F_1)=\mathbb{C}$, $C^*(E_2)=\mathcal{T}$ and $C^*(F_2)=C(S^1)$
(see the diagram~\eqref{pictpull}).

Our next example generalizes a simple gluing construction in topology. Recall that the real projective plane $\mathbb{R}{\rm P}^2$
may be represented as a~closed hemisphere with the antipodal points on the equator identified. If we further identify all those 
antipodal points, we obtain the sphere $S^2$. Here we present a $q$-deformed analog of this procedure.

The C*-algebra $C(\mathbb{R}{\rm P}^2_q)$ of the quantum real projective plane 
$\mathbb{R}{\rm P}^2_q$~\cite[Section~4]{hrsz-03}
admits a graph-algebraic presentation (see~\cite[Section~3.2]{hsz-02}) 
as the \mbox{C*-alge}\-bra of the graph given below:
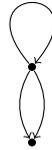
\begin{equation}
\begin{tikzpicture}[auto,swap]
\tikzstyle{vertex}=[circle,fill=black,minimum size=3pt,inner sep=0pt]
\tikzstyle{edge}=[draw,->]
\tikzstyle{cycle1}=[draw,->,out=130, in=50, loop, distance=40pt]
\tikzstyle{cycle2}=[draw,->,out=135, in=45, loop, distance=50pt]
\node[vertex] (0) at (0,0) {};
\node[vertex] (1) at (0,-1) {};
\path (0) edge[cycle1] node {} (0);
\path (0) edge[edge, bend right] node {} (1);
\path (0) edge[edge, bend left] node {} (1);
\end{tikzpicture}
\end{equation}

Due to Theorem~\ref{main} and $C(S^1)\cong C(\mathbb{R}{\rm P}^1)$, we have the following pullback diagram:
\begin{equation}\label{rp2}
\begin{gathered}
\xymatrix{
& C(S^2_{q0}) \ar[ld] \ar[rd] \\
\mathbb{C}  \ar[rd]
& 
& C(\mathbb{R}{\rm P}^2_q). \ar[ld]\\
& C(\mathbb{R}{\rm P}^1)
}
\end{gathered}
\end{equation}
Observe that the diagram~\eqref{rp2} reflects the aforementioned procedure of shrinking the copy of $\mathbb{R}{\rm P}^1$ inside 
$\mathbb{R}{\rm P}^2$ to a point.

\subsection{The quantum teardrop \boldmath$\mathbb{W}{\rm P}^1_q(1,2)$}
The classical teardrop $\mathbb{W}{\rm P}^1(1,2)$ 
may be represented as the wedge of two spheres, namely we have the following
pushout diagram:
\begin{equation}\label{wp1}
\begin{gathered}
\xymatrix{
& \mathbb{W}{\rm P}^1(1,2) \\
\{\ast\} \ar[ru]
& 
& S^2\ar[lu]~.\\
& S^1 \ar[lu] \ar[ru]
}
\end{gathered}
\end{equation}

To obtain a noncommutative counterpart of the diagram~\eqref{wp1}, we need to introduce a different kind of a noncommutative
sphere. The C*-algebra $C(S^2_{q\infty})$ of the equatorial Podle\'s quantum sphere $S^2_{q\infty}$~\cite[(3b)]{pod-87} 
admits a graph-algebraic presentation (see~\cite[Section~3.1]{hsz-02}) 
as the \mbox{C*-alge}\-bra of the graph given below.
\begin{equation}
\begin{tikzpicture}[auto,swap]
\tikzstyle{vertex}=[circle,fill=black,minimum size=3pt,inner sep=0pt]
\tikzstyle{edge}=[draw,->]
\tikzstyle{cycle1}=[draw,->,out=130, in=50, loop, distance=40pt]
\tikzstyle{cycle2}=[draw,->,out=135, in=45, loop, distance=65pt]
\node[vertex] (0) at (0,0) {};
\node[vertex] (1) at (1,-1) {};
\node[vertex] (2) at (-1,-1) {};
\path (0) edge[cycle1] node {} (0);
\path (0) edge[edge] node {} (1);
\path (0) edge[edge] node {} (2);
\end{tikzpicture}
\end{equation}
Theorem~\ref{main} applies and we obtain the  pullback diagram
\begin{equation}\label{wp1q}
\begin{gathered}
\xymatrix{
& C(\mathbb{W}{\rm P}^1_q(1,2)) \ar[ld] \ar[rd] \\
\mathbb{C}  \ar[rd]
& 
& C(S^2_{q\infty}), \ar[ld]\\
& C(S^1)
}
\end{gathered}
\end{equation}
which can be regarded as a noncommutative deformation of the diagram~\eqref{wp1}.

\subsection{The quantum complex  projective spaces \boldmath$\mathbb{C}{\rm P}^n_q$}
The CW-complex decomposition of complex projectives spaces may be described
in terms of pushout diagrams
\begin{equation}\label{cpn}
\begin{gathered}
\xymatrix{
& \mathbb{C}{\rm P}^n \\
\mathbb{C}{\rm P}^{n-1} \ar[ru]
& 
& B^{2n}\ar[lu]~.\\
& S^{2n-1} \ar[lu] \ar[ru]
}
\end{gathered}
\end{equation}
Let us recall the graph-algebraic presentation of $q$-deformations of the spaces in the diagram~\eqref{cpn}.
\begin{itemize}
\item 
The C*-algebra $C(\mathbb{C}{\rm P}^n_q)$ of the quantum complex  projective space $\mathbb{C}{\rm P}^n_q$~\cite{vs90}
is the graph C*-algebra of a graph that, for $n=3$, is given below~(see~\cite[Section~4.3]{hsz-02}):
\begin{equation}\label{qcpn}
\begin{tikzpicture}[auto,swap]
\tikzstyle{vertex}=[circle,fill=black,minimum size=3pt,inner sep=0pt]
\tikzstyle{edge}=[draw,->]
\tikzstyle{cycle1}=[draw,->,out=130, in=50, loop, distance=40pt]
\tikzstyle{cycle2}=[draw,->,out=135, in=45, loop, distance=65pt]
\node[vertex] (0) at (0,0) {};
\node[vertex] (1) at (1.5,0) {};
\node[vertex] (2) at (3,0) {};
\node[vertex] (3) at (4.5,0) {};
\path (0) edge[edge, above] node {$(\infty)$} (1);
\path (1) edge[edge, above] node {$(\infty)$} (2);
\path (0) edge[edge, bend right=50,above] node[xshift=-5,yshift=-1] {$(\infty)$} (2);
\path (1) edge[edge, bend right=50,above] node[xshift=5,yshift=-1] {$(\infty)$} (3);
\path (2) edge[edge,above] node {$(\infty)$} (3);
\path (0) edge[edge,bend right=60,above] node[yshift=-2] {$(\infty)$} (3);
\end{tikzpicture}
\end{equation}
\item 
The C*-algebra $C(B^{2n}_q)$ of the  Hong--Szyma\'nski quantum even-dimensional ball $B^{2n}_q$~\cite{hsz-08} is 
the graph C*-algebra of a graph that, for $n=3$, is given below~(see~\cite[Section~3.1]{hsz-08}):
\begin{equation}\label{qcpn}
\begin{tikzpicture}[auto,swap]
\tikzstyle{vertex}=[circle,fill=black,minimum size=3pt,inner sep=0pt]
\tikzstyle{edge}=[draw,->]
\tikzstyle{cycle1}=[draw,->,out=130, in=50, loop, distance=40pt]
\tikzstyle{cycle2}=[draw,->,out=135, in=45, loop, distance=65pt]
\node[vertex] (0) at (0,0) {};
\node[vertex] (1) at (1.5,0) {};
\node[vertex] (2) at (3,0) {};
\node[vertex] (3) at (4.5,0) {};
\path (0) edge[cycle1] node {} (0);
\path (1) edge[cycle1] node {} (1);
\path (2) edge[cycle1] node {} (2);
\path (0) edge[edge, above] node {} (1);
\path (1) edge[edge, above] node {} (2);
\path (0) edge[edge, bend right=50,above] node[xshift=-5,yshift=-1] {} (2);
\path (1) edge[edge, bend right=50,above] node[xshift=5,yshift=-1] {} (3);
\path (2) edge[edge,above] node {} (3);
\path (0) edge[edge,bend right=60,above] node[yshift=-2] {} (3);
\end{tikzpicture}
\end{equation}
\item 
The C*-algebra $C(S^{2n-1}_q)$ of the Vaskman--Soibelman  quantum odd-dimensional sphere 
$S^{2n-1}_q$~\cite[Definition on p.~106]{vs90} is the graph C*-algebra of a graph that, for $n=4$, 
 is given below (see~\cite[Section~4.1]{hsz-02}):
\begin{equation}
\begin{tikzpicture}[auto,swap]
\tikzstyle{vertex}=[circle,fill=black,minimum size=3pt,inner sep=0pt]
\tikzstyle{edge}=[draw,->]
\tikzstyle{cycle1}=[draw,->,out=130, in=50, loop, distance=40pt]
\tikzstyle{cycle2}=[draw,->,out=135, in=45, loop, distance=65pt]
\node[vertex] (0) at (0,0) {};
\node[vertex] (1) at (1.5,0) {};
\node[vertex] (2) at (3,0) {};
\node[vertex] (3) at (4.5,0) {};
\path (0) edge[cycle1] node {} (0);
\path (1) edge[cycle1] node {} (1);
\path (2) edge[cycle1] node {} (2);
\path (3) edge[cycle1] node {} (3);
\path (0) edge[edge, above] node {} (1);
\path (1) edge[edge, above] node {} (2);
\path (0) edge[edge, bend right=50,above] node[xshift=-5,yshift=-1] {} (2);
\path (1) edge[edge, bend right=50,above] node[xshift=5,yshift=-1] {} (3);
\path (2) edge[edge,above] node {} (3);
\path (0) edge[edge,bend right=60,above] node[yshift=-2] {} (3);
\end{tikzpicture}
\end{equation}
\end{itemize}
Applying Theorem~\ref{main}, we obtain the pullback diagram
\begin{equation}\label{cpnpull}
\begin{gathered}
\xymatrix{
&
C(\mathbb{C}{\rm P}^n_q)
\ar[rd] \ar[ld]&\\
C(\mathbb{C}{\rm P}^{n-1}_q) \ar[rd]
& & 
C(B^{2n}_q)~.
\ar[ld]
\\
&
C(S^{2n-1}_q)
&
}
\end{gathered}
\end{equation}
Note that the diagram~\eqref{cpnpull} was obtained in~\cite[Proposition~4.1]{adht-18} using equivariant pullback structures.

\subsection{The quantum teardrops \boldmath$\mathbb{W}{\rm P}^1_q(1,n)$}
Let $n\in\mathbb{N}\setminus\{0\}$. Consider the following graph $W_n$:
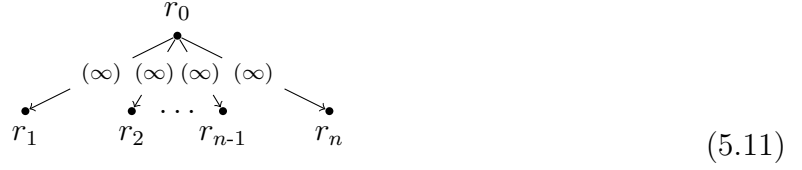
\begin{equation}\label{wn}
\begin{tikzpicture}[auto,swap]
\tikzstyle{vertex}=[circle,fill=black,minimum size=3pt,inner sep=0pt]
\tikzstyle{edge}=[draw,->]
\tikzstyle{cycle1}=[draw,->,out=130, in=50, loop, distance=40pt]
\tikzstyle{cycle2}=[draw,->,out=135, in=45, loop, distance=65pt]
\node[vertex,label=above:$r_0$] (0) at (0,0) {};
\node (1) at (-1,-0.5) {{\tiny $(\infty)$}};
\node[vertex,label=below:$r_{1}$] (2) at (-2,-1) {};
\node (3) at (-0.3,-0.5) {{\tiny $(\infty)$}};
\node[vertex,label=below:$r_{2}$] (4) at (-0.6,-1) {};
\node (10) at (0,-1) {$\ldots$};
\node (5) at (0.3,-0.5) {{\tiny $(\infty)$}};
\node[vertex,label=below:$r_{n\text{-}1}$] (6) at (0.6,-1) {};
\node (7) at (1,-0.5) {{\tiny $(\infty)$}};
\node[vertex,label=below:$r_{n}$] (8) at (2,-1) {};
\path (0) edge[draw] node {} (1);
\path (1) edge[edge] node {} (2);
\path (0) edge[draw] node {} (3);
\path (3) edge[edge] node {} (4);
\path (0) edge[draw] node {} (5);
\path (5) edge[edge] node {} (6);
\path (0) edge[draw] node {} (7);
\path (7) edge[edge] node {} (8);
\end{tikzpicture}
\end{equation}
Observe that $C^*(W_1)\cong C(S^2_{q0})$.
Moreover, one can  show (see~\cite[Section~3]{bsz-18}) that, in general, the graph \mbox{C*-al}ge\-bra 
$C^*(W_n)$ is isomorphic with the C*-algebra $C(\mathbb{W}{\rm P}^1_q(1,n))$~\cite[Section~3]{bf-12}.
We will also need the following $n$-sink extension $R^n_m$ of the graph~\eqref{cuntz}:
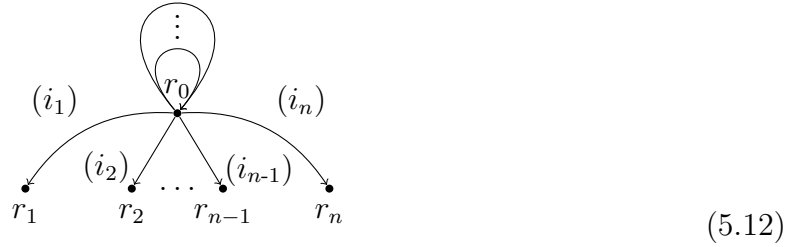
\begin{equation}\label{rnm}
\begin{tikzpicture}[auto,swap]
\tikzstyle{vertex}=[circle,fill=black,minimum size=3pt,inner sep=0pt]
\tikzstyle{edge}=[draw,->]
\tikzstyle{cycle1}=[draw,->,out=130, in=50, loop, distance=40pt]
\tikzstyle{cycle2}=[draw,->,out=130, in=50, loop, distance=70pt]
   
\node[vertex,label=above:$r_0$] (0) at (0,0) {};
%\node[vertex,label=right:$r_2$] (1) at (2,0) {};
\node (2) at (0,1.25) {\vdots};
%\node[vertex,label=below:$v$] (2) at (-1,-1) {};

\node[vertex,label=below:$r_1$] (2) at (-2,-1) {};
%\node (3) at (-0.3,-0.5) {{\tiny $(\infty)$}};
\node[vertex,label=below:$r_2$] (4) at (-0.6,-1) {};
\node (10) at (0,-1) {$\ldots$};
%\node (5) at (0.3,-0.5) {{\tiny $(\infty)$}};
\node[vertex,label=below:$r_{n-1}$] (6) at (0.6,-1) {};
%\node (7) at (1,-0.5) {{\tiny $(\infty)$}};
\node[vertex,label=below:$r_n$] (8) at (2,-1) {};
\path (0) edge[edge,bend right] node {$(i_1)$} (2);
%\path (1) edge[edge] node {} (2);
\path (0) edge[edge,below left] node[xshift=-5.0,yshift=4.0] {$(i_2)$} (4);
%\path (3) edge[edge] node {} (4);
\path (0) edge[edge, below right] node[xshift=5.0,yshift=3.0] {$(i_{n\text{-}1})$} (6);
%\path (5) edge[edge] node {} (6);
\path (0) edge[edge,bend left,above right] node {$(i_n)$} (8);

\path (0) edge[cycle1] node {} (0);
%\path (0) edge[edge] node {$f_{n+1}$} (1);
%\path (0) edge[edge,bend right] node {} (1);
%\path (0) edge[edge] node {} (2);
\path (0) edge[cycle2] node[above] {} (0);
%\path (2) edge[edge] node {} (1);

\end{tikzpicture}
\end{equation}
Here the notation $(i_j)$ means that there are $i_j\in\mathbb{N}\setminus\{0\}$ many edges from $r_0$ to $r_j$.
Now,
due to Theorem~\ref{main}, we obtain the pullback diagram
\begin{equation}\label{wpcuntz}
\begin{gathered}
\xymatrix{
&
C(\mathbb{W}{\rm P}^1_q(1,n))
\ar[rd] \ar[ld]&\\
\mathbb{C} \ar[rd]
& & 
C^*(R^n_m)~.
\ar[ld]
\\
&
\mathcal{O}_m
&
}
\end{gathered}
\end{equation}

Let us now consider the graph $G^n$ defined as a pushout of $W_n$~(see \eqref{wn}) and an another graph $H$ over the
sinks of $W_n$. The only restriction on the graph $H$ is that there exists an inclusion 
$\{r_{1},\ldots,r_{n}\}\subseteq H^0$. The graph $G^n$ is represented pictorially as follows:
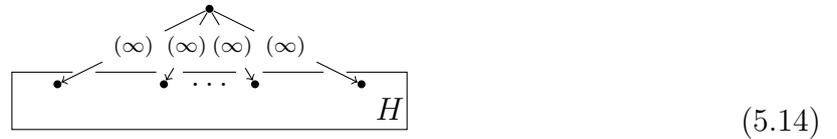
\begin{equation}
\begin{tikzpicture}[auto,swap]
\tikzstyle{vertex}=[circle,fill=black,minimum size=3pt,inner sep=0pt]
\tikzstyle{edge}=[draw,->]
\tikzstyle{cycle1}=[draw,->,out=130, in=50, loop, distance=40pt]
\tikzstyle{cycle2}=[draw,->,out=135, in=45, loop, distance=65pt]
\node[vertex] (0) at (0,0) {};
\node (1) at (-1,-0.5) {{\tiny $(\infty)$}};
\node[vertex] (2) at (-2,-1) {};
\node (3) at (-0.3,-0.5) {{\tiny $(\infty)$}};
\node[vertex] (4) at (-0.6,-1) {};
\node (10) at (0,-1) {$\ldots$};
\node (5) at (0.3,-0.5) {{\tiny $(\infty)$}};
\node[vertex] (6) at (0.6,-1) {};
\node (7) at (1,-0.5) {{\tiny $(\infty)$}};
\node[vertex,label=below right:$H$] (8) at (2,-1) {};
\path (0) edge[draw] node {} (1);
\path (1) edge[edge] node {} (2);
\path (0) edge[draw] node {} (3);
\path (3) edge[edge] node {} (4);
\path (0) edge[draw] node {} (5);
\path (5) edge[edge] node {} (6);
\path (0) edge[draw] node {} (7);
\path (7) edge[edge] node {} (8);
\draw (-1.8,-0.84) -- (-2.6,-0.84) -- (-2.6,-1.6) -- (2.6,-1.6) -- (2.6,-0.84) -- (1.8,-0.84);
\draw (-1.5,-0.84) -- (-0.7,-0.84);
\draw (-0.35,-0.84) -- (0.35,-0.84);
\draw (1.5,-0.84) -- (0.7,-0.84);
\end{tikzpicture}
\end{equation}
Next, we consider an analogous construction for the graph $R^n_m$~(see \eqref{rnm}) using the same graph $H$, 
and we denote the resulting graph by $E^n_m$. The graph $E^n_m$ is represented pictorially as follows:
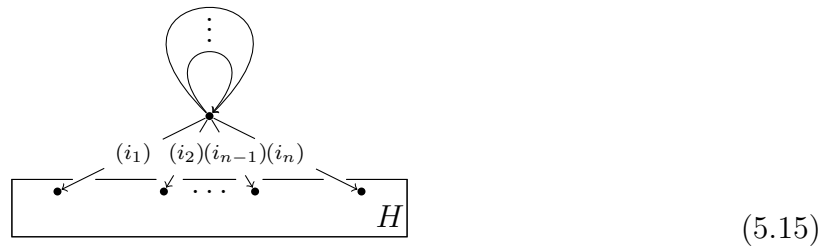
\begin{equation}
\begin{tikzpicture}[auto,swap]
\tikzstyle{vertex}=[circle,fill=black,minimum size=3pt,inner sep=0pt]
\tikzstyle{edge}=[draw,->]
\tikzstyle{cycle1}=[draw,->,out=130, in=50, loop, distance=40pt]
\tikzstyle{cycle2}=[draw,->,out=135, in=45, loop, distance=75pt]
\node[vertex] (0) at (0,0) {};
\node (1) at (-1,-0.5) {{\tiny$(i_1)$}};
\node[vertex] (2) at (-2,-1) {};
\node (3) at (-0.3,-0.5) {{\tiny$(i_2)$}};
\node[vertex] (4) at (-0.6,-1) {};
\node (10) at (0,-1) {$\ldots$};
\node (5) at (0.3,-0.5) {{\tiny$\;(i_{n-1})$}};
\node[vertex] (6) at (0.6,-1) {};
\node (7) at (1,-0.5) {{\tiny$(i_n)$}};
\node[vertex,label=below right:$H$] (8) at (2,-1) {};
\path (0) edge[draw] node {} (1);
\path (1) edge[edge] node {} (2);
\path (0) edge[draw] node {} (3);
\path (3) edge[edge] node {} (4);
\path (0) edge[draw] node {} (5);
\path (5) edge[edge] node {} (6);
\path (0) edge[draw] node {} (7);
\path (7) edge[edge] node {} (8);
\draw (-1.8,-0.84) -- (-2.6,-0.84) -- (-2.6,-1.6) -- (2.6,-1.6) -- (2.6,-0.84) -- (1.8,-0.84);
\draw (-1.5,-0.84) -- (-0.7,-0.84);
\draw (-0.35,-0.84) -- (0.35,-0.84);
\draw (1.5,-0.84) -- (0.7,-0.84);

\draw (-1.8,-0.84) -- (-2.6,-0.84) -- (-2.6,-1.6) -- (2.6,-1.6) -- (2.6,-0.84) -- (1.8,-0.84);
\draw (-1.5,-0.84) -- (-0.7,-0.84);
\draw (-0.35,-0.84) -- (0.35,-0.84);
\draw (1.5,-0.84) -- (0.7,-0.84);
\node (11) at (0,1.25) {\vdots};
\path (0) edge[cycle1] node {} (0);
\path (0) edge[cycle2] node[above] {} (0);
\end{tikzpicture}
\end{equation}
Theorem~\ref{le.blob} applies and, for any $n,m\in\mathbb{N}\setminus\{0\}$, we obtain the following pullback diagram:
\begin{equation}\label{wpcuntz2}
\begin{gathered}
\xymatrix{
&
C^*(G^n)
\ar[rd] \ar[ld]&\\
\mathbb{C} \ar[rd]
& & 
C^*(E^n_m)~.
\ar[ld]
\\
&
\mathcal{O}_m
&
}
\end{gathered}
\end{equation}

\section*{Acknowledgement}\noindent
The work on this project was  partially supported by NCN grant 2015/19/B/ST1/03098
(Piotr M.\ Hajac, Mariusz Tobolski) and by NSF grant DMS-1801011 (Alexandru Chirvasitu).
It is a pleasure to thank Sarah Reznikoff for a helpful discussion. P.M.H.\ is also grateful to SUNY Buffalo for its hospitality and financial support.

\end{document}